\documentclass[a4paper,11 pt]{amsart}
\usepackage{hyperref}
 
\usepackage[usenames,dvipsnames,svgnames,table]{xcolor}
\usepackage[all]{xy}
 \usepackage{enumitem}
 \usepackage[normalem]{ulem}

\newcommand{\tsk}[1]{\textcolor{YellowOrange}}

\usepackage{tikz}
\usepackage{tikz-cd}

\makeatletter
\def\@endtheorem{\endtrivlist}
\makeatother

\usepackage[active]{srcltx}
\usepackage{amsmath}
\usepackage{amssymb}
\usepackage{amscd}
\usepackage{amsthm}
\usepackage[latin1]{inputenc}
\usepackage{mathrsfs}  
\usepackage{comment}

\usepackage{nicefrac}

\newtheorem{teo}{Theorem}[section]
\newtheorem{defin}[teo]{Definition}
\newtheorem{prop}[teo]{Proposition}

\newtheorem{lemma}[teo]{Lemma}
\theoremstyle{definition}

\newtheoremstyle{dico}
 {\baselineskip}   
  {\topsep}   
  {}  
  {0pt}       
  {} 
  {.}         
  {5pt plus 1pt minus 1pt} 
  {}          
\theoremstyle{dico}

\numberwithin{equation}{section}

\newcommand{\C}{\mathbb{C}}

\newcommand{\Z}{{\mathbb{Z}}}

\newcommand{\Pic}{\operatorname{Pic}}
\newcommand{\Spec}{\operatorname{\textit{Spec} }}
\newcommand{\Alb}{\operatorname{Alb}}

\renewcommand{\phi}{\varphi}

\renewcommand{\phi}             {\varphi}

 \renewcommand{\Im}              {\operatorname{Im}}

\begin{document}

\author{Paola Porru, Sammy Alaoui Soulimani}

\title[Divisors of  $\mathcal{A}^{(1,1,2,2)}_4$]{Divisors of  $\mathcal{A}^{(1,1,2,2)}_4$}

\address{Universit\`a di Pavia \newline Dipartimento di Matematica \newline Via Ferrata 1 \newline 27100 Pavia}
\email{paola.porru01@ateneopv.it} 
\address{Universitetet i Stavanger \newline institutt for matematikk og naturvitenskap \newline Kjell Arholmsgate 41 \newline 4036 Stavanger} \email{sammy.a.soulimani@uis.no}


\maketitle

\begin{abstract}
We construct two divisors in the moduli space $\mathcal{A}_4 ^{(1,1,2,2)}$ and we check their invariance and non-invariance under the canonical involution introduced by Birkenhake and Lange~\cite{birkenhake2003isomorphism}.
 \end{abstract}

\setcounter{tocdepth}{1}
\tableofcontents{}

\section{Introduction}

The purpose of this paper is to study the geometry of the moduli space $\mathcal{A}_4^{(1,1,2,2)}$, parametrizing isomorphism classes of $4$-dimensional abelian varieties with polarization of type $(1,1,2,2)$. More precisely, our aim is to study its Picard group $\Pic(\mathcal{A}_4^{(1,1,2,2)})$ in order to get information about its Kodaira dimension.

In general, the problem of computing the Kodaira dimension of the moduli spaces $\mathcal{A}_g^{(d_1, \dots d_g)}$ has been a topic of intense study in the last years. Since the direct calculation of the Kodaira dimension $\kappa (X)$ of a variety $X$ (defined as the maximum of the dimension of $\varphi_{mK_X}(X)$\footnote{$\varphi_{mK_X}$ is the rational map from $X$ to the projective space associated with the linear system $|mK_X|$.} for $m \geq 1$, and $-\infty$ if $|mK_X|=\emptyset$ for all $m$) is often very hard to perform, the majority of results about $\kappa(\mathcal{A}_g^{(d_1, \dots d_g)})$ have been obtained as a consequence of generality and rationality properties: it is well known for instance that every unirational variety $X$ (i.e. that admits a rational dominant map $\mathbb{P} \dashrightarrow X$) has $\kappa (X) = -\infty$.

The case $\mathcal{A}_g := \mathcal{A}_g^{(1, \dots 1)}$ of principally polarized abelian varieties, has been almost solved: it has been shown that the moduli space $\mathcal{A}_g$ is unirational if $g \leq 5$, that implies that its Kodaira dimension $\kappa(\mathcal{A}_g)= - \infty$ (see ~\cite{clemens1983double},~\cite{donagi1984unirationality}). Is has also been shown that the moduli spaces $\mathcal{A}_g$ are of general type for $g \geq 7$, so their Kodaira dimension turns out to be maximal (see~\cite{mumford1983kodaira},~\cite{tai1982kodaira}). The only unsolved case is $\mathcal{A}_6$, whose Kodaira dimension is yet unknown. 

Less is known about the Kodaira dimension of the case of non-principally polarized abelian varieties. Concerning abelian surfaces, Hulek and Sankaran have shown that the compactification of the moduli space of abelian surfaces with a $(1,p)$-polarization and a level structure $\bar{\mathcal{A}}_p$ ($p$ a prime) is of general type for $p \geq 41$ ~\cite{hulek1994surfaces} . We also recall the result of Tai, who proved that the moduli space $\mathcal{A}_g^{(d_1, \dots d_g)}$ is of general type when $g \geq 16$ for every choice of the polarization, and when $g \geq 8$ but only for certain polarizations~\cite{tai1982kodaira}.
The only result about unirationality of such moduli spaces is due to Bardelli, Ciliberto and Verra~\cite{bardelli1995curves}, who proved that $\mathcal{A}_4^{(1,2,2,2)}$ is unirational. Moreover, since this space is isomorphic to $\mathcal{A}_4^{(1,1,1,2)}$ (see Birkenhake and Lange~\cite{birkenhake2003isomorphism}), this also implies the unirationality of $\mathcal{A}_4^{(1,1,1,2)}$. Nevertheless nothing is known about neither the unirationality of $\mathcal{A}_4^{(1,1,2,2)}$ nor its Kodaira dimension.

In this paper, in order to better understand the geometry of $\mathcal{A}_4^{(1,1,2,2)}$, we try to get more information on its Picard group. In Section \ref{construction} we construct explicit divisors of that moduli space following two different approaches: the first divisor is constructed as the image of the Prym map $P: \mathcal{R}_{2,6} \rightarrow \mathcal{A}_4 ^{(1,1,2,2)}$, sending a cover $\pi : D \rightarrow C$ in $\mathcal{R}_{2,6}$ to its Prym variety. The second divisor is constructed from the map $\tilde{\mathcal{A}_4} \rightarrow \mathcal{A}_4^{(1,1,2,2)}$, sending a principally polarized abelian variety $X$ of dimension $4$ together with a fixed totally isotropic order $4$ subgroup $H$ of $2$-torsion elements to the quotient $X/H$, and then considering the image of the Jacobian locus by this map (see section \ref{construction} for a definition of $\tilde{\mathcal{A}_4}$).

In Section \ref{SectProof}, to get more informations about these divisors, we check if they are invariant under the natural involution defined on the moduli space $\mathcal{A}_4^{(1,1,2,2)}$ by Birkenhake and Lange, sending a polarized abelian variety $(A, L_A)$ to its dual $(A^{\vee}, L_A^{\vee})$ (see~\cite{birkenhake2003isomorphism}). We almost immediately obtain that the divisor constructed with the Prym procedure is fixed by the involution, by using the result of Pantazis stating that two bigonally related covers have dual Prym varieties (see~\cite{pantazis1986prym}). On the other hand, with a bit more work, we obtain that the second divisor is not invariant under the involution: the clue here is a Theorem due to Bardelli and Pirola, stating that if there exists an isogeny between two Jacobians $JC$ and $JC'$ ($JC$ generic, with dimension at least $4$), then the two Jacobians have to be isomorphic, and the isogeny  is the multiplication by an integer (see~\cite{bardelli1989curves}). Since the involution does not preserve this divisor, we get a very explicit description of a different divisor in $\mathcal{A}_4^{(1,1,2,2)}$, obtained by duality.

\subsection*{Acknowledgements}
First and foremost, we thank professor J.C. Naranjo for suggesting this very interesting problem and for giving us proper guidance and advice. We also thank V. Gonz\'{a}lez-Alonso for the useful conversations we have had. And finally we thank professors A. Ragusa, F. Russo and G. Zappal\`{a}, the organizers of the summer school Pragmatic 2016, for providing us with the warm and rich environment at the university of Catania, which not only led to the birth of this modest work, but also made us know some amazing and remarkable people.

\section{Notation and preliminaries}\label{preliminaries}

We work over the field $\mathbb{C}$ of complex numbers. We start by stating some well known results about complex polarized abelian varieties and Pryms, then we recall the main ideas of the bigonal construction, which will be used in the next section. Our main reference for this preliminary section is Birkenhake and Lange's book~\cite{birkenhake2013complex}.

\subsection{Polarized abelian varieties}\label{modspace}

Let $(A,H)$ be a polarized abelian variety. Fix a line bundle $L \in \Pic(A) $ satisfying $c_1(L)=H$. The morphism $\lambda_L : A \longrightarrow A^\vee$ given by $a \mapsto \tau^*_a L \otimes L^{-1}$ is an isogeny. Here, $A^\vee=\Pic^0(A)$ is the dual of $A$, and $\tau^*_a$ is the translation by $a$ in $A$. We get the following result describing the kernel $K(L)$ of $\lambda_L$:
\begin{teo}
	If $L$ is a polarization of type $(d_1, \ldots, d_g)$ with $d_i \mid d_{i+1}$ for all $i=1, \ldots, g$, then 
	$$K(L)\cong \big( \Z/d_1\Z \times \ldots \times \Z/d_g\Z \big)^2.$$
\end{teo}
It is useful to note that $\deg(\lambda_L)=|K(L)|=d_1^2 \times \ldots \times d_g^2$.

In order to understand better the relation between line bundles over isogenous abelian varieties, we introduce the Riemann bilinear form associated to a line bundle: if $K(L)$ is the kernel of a line bundle $L$ over $A=V/\Lambda$, we define the Riemann bilinear form as the bilinear alternating form
\begin{equation*}\begin{split}
e^L : K(L) \times K(L) &\longrightarrow \C^*\\
(x,y) &\longmapsto \exp^{-2i\pi H(\tilde{x},\tilde{y})},
\end{split}\end{equation*}
$\tilde{x}$, $\tilde{y}$ being lifting of respectively $x$, $y$ to the vector space $V$. Also we recall that $H=c_1(L)$ and we have $K(L)=\{x \in A \, | \, H(l,\tilde{x}) \in \Z, \, \mbox{for all} \; l \in \Lambda  \}$ (see \cite{debarre2005complex} Chapter VI Section 4).

Note that if the line bundle $L$ is ample then the form $e^L$ is non-degenerate. To appreciate the importance of this pairing, we state two useful results. whose proofs can be found in Birkenhake and Lange's book . Before stating them we recall that a subgroup $K < K(L)$ is totally isotropic with respect to $e^L$ if for all $x,y \in K$ we have $e^L(x,y)=1$.


\begin{prop}\label{isogenythm}
For an isogeny $f:X \rightarrow Y$ of abelian varieties and a line bundle $L \in \Pic (X)$ the following statements are equivalent:
\begin{enumerate}
\item $L=f^*(M)$ for some $M \in \Pic(Y)$,
\item $\ker(f)$ is a totally isotropic subgroup of $K(L)$ with respect to $e^L$.
\end{enumerate}
\end{prop}

\proof
See (\cite{birkenhake2013complex}, Corollary 6.3.5).

\begin{prop}\label{riemmform}
	Let $f:X \rightarrow Y$ be a surjective morphism of abelian varieties, and let $M$ be a line bundle over $Y$. Then $e^{f^*M}(x,x')=e^M(f(x),f(x'))$ for all $x,x' \in f^{-1}(K(M))$.
\end{prop}

\proof
See (\cite{birkenhake2013complex}, Proposition 6.3.3).
%
%

To conclude this section we recall the construction of Birkenhake and Lange's involution: denote by $\mathcal{A}_g^{(d_1, \ldots, d_g)}$ the coarse moduli space parametrizing isomorphism classes of $g$-dimensional polarized abelian varieties of type $(d_1, \ldots, d_g)$; it is a quasi-projective variety of dimension $\frac{g(g+1)}{2}$. In~\cite{birkenhake2003isomorphism}, Birkenhake and Lange have shown that there is an isomorphism of coarse moduli spaces  $$\mathcal{A}_g^{(d_1, \ldots, d_g)} \xrightarrow{\cong} \mathcal{A}_g^{(\frac{d_1d_g}{d_{g}}, \frac{d_1d_g}{d_{g-1}} \ldots, \frac{d_1d_g}{d_{2}},\frac{d_1d_g}{d_1})}.$$
In the case where $g=4$ and $(1,1,2,2)$ is the polarization type, we get an automorphism $\rho : \mathcal{A}_4^{(1,1,2,2)} \longrightarrow \mathcal{A}_4^{(1,1,2,2)}$, associating to a polarized abelian variety $(A,L_A)$ its dual variety $(A^\vee,L_{A^\vee})$. The polarization $L_{A^\vee}$ on $A^\vee$  is constructed in order to satisfy $(L_{A^\vee})^\vee=L_A$ (see~\cite{birkenhake2003isomorphism}, Proposition 2.3). Since $(A^\vee)^\vee=A$ we get $\rho^2((A,L_A))=(A,L_A)$, hence $\rho$ is an involution on the moduli space $\mathcal{A}_4^{(1,1,2,2)}$.
 
 We remind that to give a polarization on an abelian variety $A$ is equivalent to give an isogeny $\lambda : A \longrightarrow A^{\vee}$ with $\lambda = \lambda^{\vee}$.
 
\subsection{Prym maps and Prym varieties}\label{Prym_theory}

Let $C \in \mathcal{M}_g, D \in \mathcal{M}_{g'}$, and let $D \xrightarrow{\pi} C$ be a finite morphism of degree $d$ branched on a divisor $B=q_1+\ldots +q_r$, with $q_i \in C$ and $q_i \neq q_j$ for all $i \neq j$. If $B$ is nonzero, we call the morphism $\pi$ a branched covering of $C$ of degree $d$. 

 Even though what follows can be defined for a general degree $d$, we shall focus on the case $d=2$, which is of interest for us. In this case, the curve $D$ is obtained as $\Spec \, (\mathcal{O}_C \oplus \eta^{-1})$ with $\eta \in \Pic(C)$ such that $\eta^{\otimes 2} \cong \mathcal{O}_C(B)$. Observe that if $\pi$ is a degree-$2$ cover which ramifies over $r$ points of $C$, then the Riemann-Hurwitz formula gives that the genus of $D$ is $2g-1+\frac{r}{2}$ (note that $r$ has to be even). 
%
%

Let $Nm_{\pi} : JD \longrightarrow JC$ be the norm map. We remind that it is surjective, and its kernel is connected when the covering $\pi$ is branched, otherwise it has two components.

 We are ready to define the Prym variety attached to a cover.
\begin{defin}
The Prym variety attached to the cover $D \xrightarrow{\pi} C$ is the connected component containing the origin of the kernel of the norm map: $$P(D,C)=\ker(Nm_\pi)^0.$$
\end{defin}

The Prym variety $(P(D,C), \varXi)$ turns out to be a polarized abelian variety of dimension $g-1+\frac{r}{2}$ : the polarization $\varXi$ is obtained as the first Chern class of the restriction on $P(D,C)$ of the line bundle $\mathcal{O}_{JD}(\Theta_D)$, where $\Theta_D$ is the principal polarization of $JD$. Note that in the case of $d=2$, $\varXi$ is of type $\underbrace{(1, \ldots , 1,}_{\frac{r}{2}-1}$ $\underbrace{2, \ldots ,2)}_{g}$.
%
%
%

\subsection{The bigonal construction} 

The bigonal construction is a procedure that associates to a tower of double coverings $D \rightarrow C \rightarrow K$ another tower of double coverings $\Gamma \longrightarrow \Gamma_0 \longrightarrow K$, such that $P(\Gamma,\Gamma_0)$ is the dual of $P(D,C)$. Since the duality result of Prym varieties will be useful later in our discussion, we give some details (see Pantazis for an accurate description~\cite{pantazis1986prym}).

Let $\varphi : C \rightarrow K $ be a covering of degree $2$ (hence the "bi" in bigonal) and $\pi : D \rightarrow C$ be a branched covering of degree $2$. The curve $D$ is equipped with an involution $\iota$ which exchanges the two elements of the fiber over a generic point $c \in C$. The two given coverings determine a degree $2^2$ covering $\Gamma \longrightarrow K$, whose fiber over a generic point $k \in K$ consists of $4$ sections $s_k$ of $\pi$ over $k$ :
$$s_k : \varphi^{-1}(k) \longrightarrow \pi^{-1}\varphi^{-1}(k), \qquad \qquad \pi \circ s_k = id_K.$$
 
Now observe that there is an immersion of $K$ in $C^{(2)}$ sending $x \in K$ to $\varphi^{-1}(x) \in C^{(2)}$. The curve $\Gamma$ is then defined as the pre-image of $K$ by the map $\varphi^{(2)} : D^{(2)} \longrightarrow C^{(2)}$.

Note that we can view a point $p \in \Gamma$ which belongs to the fiber of some $k \in K$ as a section $s_k$. 
 
There is an involution on $\Gamma$ defined by $\tilde{\iota}(s_k)=\iota \circ s_k$, $k \in K$, which in turn gives an equivalence relation where two points $s_1,s_2 \in \Gamma$ are said to be equivalent if $s_1=\tilde{\iota}(s_2)$. Considering the quotient $\Gamma_0=\Gamma/\tilde{\iota}$ we obtain a tower of degree $2$ coverings $\Gamma \longrightarrow \Gamma_0 \longrightarrow K$.

The two towers $D \xrightarrow{\pi} C \xrightarrow{\varphi} K$ and $\Gamma \xrightarrow{\tilde{\pi}} \Gamma_0 \xrightarrow{\tilde{\varphi}} K$ are said to be bigonally related (see Donagi for details~\cite{donagi1992fibers}). Since $\varphi$ and $\pi$ are branched, this implies that $\tilde{\varphi}$ and $\tilde{\pi}$ are branched as well.

For $K=\mathbb{P}^1$, we have the following result due to Pantazis~\cite{pantazis1986prym}: 
\begin{teo}\label{Pantazis}
	Consider a pair of maps of degree $2$, $D \rightarrow C \rightarrow \mathbb{P}^1$, and the bigonally related tower
	$\Gamma \rightarrow \Gamma _0 \rightarrow \mathbb{P}^1$.
	Consider then the Pryms:
	\begin{equation*}\begin{split}
	P(D,C) := \ker(Nm: J(D) \rightarrow J(C))^0, \\
	P(\Gamma,\Gamma_0) := \ker(Nm: J(\Gamma) \rightarrow J(\Gamma _0))^0.
	\end{split}\end{equation*}
	Then $(P(D,C), \, \Xi \, )$ and $(P(\Gamma,\Gamma_0), \, \Xi ' \, )$ are dual as polarized abelian varieties.
\end{teo}

We conclude this introductory section by briefly defining some notions and fixing some notation which we shall use throughout the rest of this work :

\begin{itemize}
\item $X_m < X$ is the kernel of $\cdot m : X \longrightarrow X$, the multiplication by $m$. We will usually refer to $X_m$ as the subgroup of $m$-torsion elements of $X$.

\item $\mathcal{R}_{g,r}$ will denote the moduli space of double coverings of a curve of genus $g$ with $r$ ramification points.


\item we denote as $P: \mathcal{R}_{g,r} \rightarrow \mathcal{A}^{\delta}_{g-1-\frac{r}{2}}$ the Prym map, associating to a covering its Prym variety.

\item if $C$ is a curve, $\Theta_C$ will denote the principal polarization of the Jacobian $JC$. If $A$ is a polarized abelian variety, we will use the line bundle $L_A$ to refer to the polarization of $A$, instead of the hermitian form $H=c_1(L_A)$. 
\end{itemize}

\section{Construction of divisors in $\mathcal{A}_4^{(1,1,2,2)}$}\label{construction}
In this section we construct two divisors of the moduli space $\mathcal{A}_4^{(1,1,2,2)}$: the first one will be constructed as the closure of the image of $\mathcal{R}_{2,6}$ by the Prym map $P$, the other one will be obtained as the image of $\mathcal{M}_g$ in $\mathcal{A}_4^{(1,1,2,2)}$ via the Torelli map and a quotient construction.

\subsection{Prym construction}
The first construction of a divisor in $\mathcal{A}_4^{(1,1,2,2)}$ immediately follows  from the Prym map $P: \mathcal{R}_{2,6} \rightarrow \mathcal{A}_4 ^{(1,1,2,2)}$ which sends a covering $\pi : D \rightarrow C$ in $\mathcal{R}_{2,6}$ to its Prym variety. From the general theory of Pryms, we get that since the cover $\pi$ ramifies, the kernel of the norm map is connected, thus $P(D,C) = \ker \lbrace Nm(\pi): JD \rightarrow JC \rbrace$. We obtain a Prym variety of dimension $4$ and polarization of type $(1,1,2,2)$.

The Prym map $P: \mathcal{R}_{2,6} \rightarrow \mathcal{A}_4 ^{(1,1,2,2)}$ has been studied in a recent work of J. C. Naranjo and A. Ortega~\cite{naranjo2016degree}: the two authors show that it is injective. Since $\mathcal{R}_{2,6}$ has dimension $3g -3 +r=9$, the closure of its image by $P$ is a divisor of the $10$-dimensional moduli space $\mathcal{A}_4^{(1,1,2,2)}$. We name the obtained divisor $\mathcal{P}$.

\subsection{Quotient construction}
First, we define the following moduli space of principally polarized abelian varieties of dimension $4$ with a fixed totally isotropic subgroup of $2$-torsion elements:
\begin{equation*}\begin{aligned}
\tilde{\mathcal{A}}_4 = \lbrace (X, L_X, H) | (X,L_X) \mbox{ is a ppav of dimension $4$, } \\ H \subset X_2 \mbox{ is a totally isotropic subgroup of four elements}\rbrace / \cong.
\end{aligned}\end{equation*}
For $(X,L_X,H) \in \tilde{\mathcal{A}}_4 $, let $A:=X/H$. This gives a degree $4$ isogeny $f : X \rightarrow A$. Thanks to Proposition (\ref{isogenythm}), we can choose over $A$ a polarization $L_A$ whose pullback by $f$ is $L_X ^{\otimes 2}$. Considering the isogenies induced by the polarizations, we get the following diagram:

\[\xymatrixcolsep{5pc}\xymatrixrowsep{3pc}
\xymatrix{
X \ar[r]^f \ar[d]_{2\lambda _{X}} & A \ar[d]^{\lambda_{L_A}} \\
X^\vee  & A^\vee \ar[l]^{f^\vee}}
\]

Observe that $X_2 = f^{-1}(\ker(f^{\vee} \circ \lambda _{L_A}))$. Computing the degree of the involved maps we get that $\deg(2\lambda _{X}) = | X_2 | = 2^8$ has to be equal to $\deg(f^\vee \circ \lambda_{L_A} \circ f) = 2^2 \cdot | \ker(\lambda_{L_A}) | \cdot 2^2 $, meaning that $| \ker(\lambda _{L_A})| = 2^4$. Observe as well that $\ker(\lambda _{L_A})$ is a commutative subgroup of $A_2$ and therefore all of its elements have order two. Then   $$\ker(\lambda _{L_A}) \cong (\mathbb{Z}/2\mathbb{Z} \times \mathbb{Z}/2\mathbb{Z})^2,$$ meaning that $(A, L_A) \in \mathcal{A}_4^{(1,1,2,2)}$. 

This construction gives a finite covering $ \varpi: \tilde{\mathcal{A}}_4 \longrightarrow \mathcal{A}^{(1,1,2,2)}_4$ which sends a triple $(X,L_X,H)$ to $(A,L_A)$.

Now let us consider the Torelli map $\tau : \mathcal{M}_g \rightarrow \mathcal{A}_g$, associating to every smooth curve of genus $g \geq 3$ its Jacobian as a principally polarized abelian variety. This map is well known to be injective (Torelli Theorem), so the closure of its image, called the Jacobian (or Torelli) locus, is a subvariety of $\mathcal{A}_g$ of dimension $3g-3$. Focusing on our case of interest, which is for $g=4$, we obtain that the Jacobian locus is actually a sub-variety of dimension $9$ of $\mathcal{A}_4$. Hence, we can consider the Jacobian locus inside $\tilde{\mathcal{A}_4}$ in the natural way and its image by the finite covering $\varpi$ defines a divisor in $\mathcal{A}_4^{(1,1,2,2)}$. We shall call it $\mathcal{J}$.



\section{The main Theorem}\label{SectProof}

In the previous section we have obtained the divisor $\mathcal{P}$ via the Prym construction, and the divisor $\mathcal{J}$ obtained from the Jacobian locus thanks to the quotient construction. In this section we ask how do these two divisors behave under the involution $\rho$. We state here our main result:

\begin{teo}\label{main_thm}
Let $\mathcal{P}$ and $\mathcal{J}$ be the divisors of $\mathcal{A}_4^{(1,1,2,2)}$ constructed in Section \ref{construction}. Let $\rho : \mathcal{A}_4^{(1,1,2,2)} \rightarrow \mathcal{A}_4^{(1,1,2,2)}$ be the Birkenhake and Lange's involution. Then we have the following:
\begin{enumerate}
\item $\mathcal{P} = \rho (\mathcal{P})$: $\mathcal{P}$ is invariant under the involution;
\item $\mathcal{J} \neq \rho(\mathcal{J})$: $\mathcal{J}$ is not invariant under the involution.
\end{enumerate}
\end{teo}
\begin{proof}[Proof of point (1)]
To prove point (1) of Theorem (\ref{main_thm}), we need to show that the dual of a Prym variety inside the Prym divisor $\mathcal{P}$ is also a Prym variety. This follows from the bigonal construction and Theorem (\ref{Pantazis}). In fact, let $D \xrightarrow{\pi} C$ be a general  branched covering in $\mathcal{R}_{2,6}$. $C$ is a hyperelliptic curve since has genus two, so we get a degree two covering $ C \xrightarrow{\varphi} \mathbb{P}^1$. The ramification points of $\varphi$ are the six Weierstrass points on $C$, which by generality we can suppose to be different than the branch locus of the covering $D \xrightarrow{\pi} C$. Applying the bigonal construction to the tower $D \xrightarrow{\pi} C \xrightarrow{\varphi} \mathbb{P}^1$, we get a corresponding tower $\Gamma \xrightarrow{\tilde{\pi}} \Gamma_0 \xrightarrow{\tilde{\varphi}} \mathbb{P}^1$, where $\Gamma \xrightarrow{\tilde{\pi}} \Gamma_0$ is a degree two covering with $6$ branch points. Now we need to see that $P(\Gamma,\Gamma_0)$ is in $\mathcal{P}$. Let us count the genera of the curves $\Gamma$ and $\Gamma_0$: the ramification divisor of the degree $4$ covering $\Gamma \longrightarrow \mathbb{P}^1$ is 
$$R=w_1 + \ldots + w_6 + b_1 + \ldots +b_6 + b'_1 + \ldots + b'_6,$$ 
where $w_i$ is in the fiber over $k_{w_i} \in \mathbb{P}^1$, which is the image of a Weierstrass point by $\varphi$, whereas $b_i$'s and $b'_i$'s are the elements of the fiber over
$k_{b_i} \in \mathbb{P}^1$ which is the image by $\varphi$ of a branch point of $\pi$. Hence $\deg(R)=18$, so by the Riemann-Hurwitz formula, we get that the genus of $\Gamma$ is $6$. Now the ramification divisor of the degree $2$ covering $\Gamma \xrightarrow{\gamma} \Gamma_0$ is 
$$R'=w_1 + \ldots + w_6,$$
the points $w_i$ are as described above and are those fixed by the involution $\tilde{\iota}$, so $\deg(R')=6$. Using the Riemann-Hurwitz formula we get that the genus of $\Gamma_0$ is $2$. Thus the covering $\Gamma \xrightarrow{\tilde{\pi}} \Gamma_0$ lives in $\mathcal{R}_{2,6}$, meaning that $P(\Gamma,\Gamma_0)$ is indeed in the divisor $\mathcal{P}$. Using Theorem (\ref{Pantazis}), we get that $P(D,C)$ and $P(\Gamma,\Gamma_0)$ are dual, which concludes the proof of (1).

\end{proof}

Part (2) of Theorem (\ref{main_thm}) requires more work.
 
From now on, let $(X,L_X)=(JC,\Theta_C)$ for some curve $C$, and $(A,L_A)=(\frac{JC}{\langle\alpha_1,\alpha_2\rangle},L_A)$  where $\alpha_1$ and $\alpha_2$ are $2$-torsion elements in $JC$ satisfying \\ $e^{2\Theta_C}(\alpha_1,\alpha_2)=1$, $e^{2\Theta_C}$ being the Riemann bilinear form associated to $\mathcal{O}_{JC}(2\Theta_C)$. Recall that elements in $\mathcal{J}$ are polarized abelian varieties $(A,L_A)$ with an isogeny of degree 4 from a Jacobian $f : JC \longrightarrow A$, such that $f^*(L_A)=\mathcal{O}_{JC}(2\Theta_C)$. The divisor $\mathcal{J}' = \rho (\mathcal{J})$ has to be a variety whose elements are polarized abelian varieties $(A',L_{A'})$ with an isogeny of degree 4 to a Jacobian $f' : A' \longrightarrow JC'$, such that $f'^*(\Theta_{C'})=L_{A'}$.

To find a more explicit description of $\mathcal{J}'$, the following Lemma will be useful:

\begin{lemma}\label{lemma_ker}
The kernel $K(L_A)$ is given by :
\begin{equation*}
K(L_A) = \frac{\langle \alpha _1 , \alpha _2 \rangle ^\perp}{\langle \alpha _1, \alpha _2 \rangle} \subset \frac{JC}{\langle \alpha _1 , \alpha _2 \rangle} = A,
\end{equation*}
where orthogonality is considered in $JC_2$ with respect to $e^{2\Theta_C}$.
\end{lemma}
\begin{proof}
Since both groups have the same cardinality (16 elements), it is enough to prove one inclusion. Let's see that $K(L_A)$ is contained in 
\[ \frac{ \langle \alpha _1, \alpha _2 \rangle ^\perp}{\langle \alpha _1, \alpha _2 \rangle}. \]
Let $\tilde{a} \in K(L_A)$, then $\tilde{a} = f (a)$ for some $a \in JC$. Therefore:
\[ 1 = e^{L_A}(f (a), 0) = e^{L_A} (f (a), f (\alpha _i )) = e^{2 \Theta_C} (a, \alpha _i),\] where the last equality is obtained thanks to proposition \ref{riemmform}. Hence $\tilde{a} \in \langle \alpha _1, \alpha _2 \rangle ^\perp$.
\end{proof}
Let \begin{equation*}\begin{aligned}
\tilde{\mathcal{A}}_4' = \lbrace (X, L_X, H) | (X,L) \mbox{ is a ppav of dimension 4, } \\ H \subset X_2 \mbox{ and } H^{\perp} \mbox{ is an isotropic subgroup of four elements}\rbrace / \cong.
\end{aligned}\end{equation*}

We now define the new divisor $\mathcal{J}'$ using a construction analogous to the quotient one: let $(X, L_X, H) \in \tilde{\mathcal{A}}_4'$, and let us put $A'=X/H$. This gives a degree $4$ isogeny $f' : A' \longrightarrow X/X_2 \cong X$. $A'$ is polarized by $L_{A'}=f'^*(L_X)$, which is of the desired type $(1,1,2,2)$. The moduli space $\tilde{\mathcal{A}}_4'$ also gives a finite covering $\varpi'$ for $\mathcal{A}^{(1,1,2,2)}_4$. As before, the image of the Jacobian locus by $\varpi'$ defines a divisor which is in fact $\mathcal{J}'$.

Now for $A = JC/ \langle \alpha _1, \alpha _2 \rangle$, we have that $A^\vee=\rho(A)$. Recalling that by definition $\lambda _{L_A} : A \rightarrow A^{\vee}$, then by using Lemma \ref{lemma_ker} and the third isomorphism Theorem we can write down $A^{\vee}$ explicitly as a quotient of $JC$:
\begin{equation*}
A^{\vee} \cong \frac{A}{\ker (\lambda _{L_A})} \cong\frac{JC \big{/} \langle \alpha _1, \alpha _2 \rangle}{\langle \alpha _1 , \alpha _2 \rangle ^\perp \big{/} \langle \alpha _1, \alpha _2 \rangle} \cong \frac{JC}{\langle \alpha _1, \alpha _2 \rangle ^\perp}.
\end{equation*}

Moreover, $(A'=\frac{JC}{\langle \alpha _1, \alpha _2 \rangle ^\perp},L_{A'})$ is the image of $(JC, \Theta_C, \langle \alpha _1, \alpha _2 \rangle ^\perp)$ by $\varpi'$.
This duality argument leads to the fact that the two divisors $\mathcal{J}$ and $\mathcal{J}'$ are linked by the following diagram:
\begin{center}
\begin{tikzpicture}
\matrix (m) [matrix of math nodes,row sep=3em,column sep=4em,minimum width=2em]
{
	\tilde{\mathcal{A}_4} & \tilde{\mathcal{A'}_4} \\
	\mathcal{A}^{(1,1,2,2)}_4 & \mathcal{A}^{(1,1,2,2)}_4 \\};
\path[-stealth]
(m-1-1) edge node [left] {$\varpi$} (m-2-1)
edge node [above] {$\perp$} (m-1-2)
(m-2-1.east|-m-2-2) edge node  [above] {$\rho$} (m-2-2)
(m-1-2) edge node [right] {$\varpi'$} (m-2-2)
;
\end{tikzpicture}
\end{center}

The map $\perp$ takes the triple $(X,L_X,H)$ to $(X^\vee,L_X^\vee,H^\perp)$; the maps $\varpi$, $\varpi'$ are the two finite coverings of $\mathcal{A}_4^{(1,1,2,2)}$ defined above, and the map $\rho$ is the Birkenhake and Lange involution.
 We see that the diagram commutes thanks to the following Lemma:

\begin{lemma}\label{lem2}
The pullback by $f ^\vee : A^\vee \longrightarrow JC^\vee$ of $\Theta^\vee_C$ is algebraically equivalent to $L_{A^\vee}$.
\end{lemma}
\begin{proof}
The statement is equivalent to $f \lambda _{\Theta}^{-1} f ^{\vee} = \lambda _{L_A^\vee}$. Let us first show that $f \lambda _{\Theta}^{-1}  f ^{\vee} \lambda_{L_A} f$ is equal to $\lambda _{L_A^\vee}  \lambda_{L_A} f=2f$. 
We know that $f ^{\vee} \lambda_{L_A}  f=2\lambda_{\Theta}$, therefore $f \lambda _{\Theta}^{-1} f ^{\vee} \lambda_{L_A} f= f \lambda _{\Theta}^{-1} 2\lambda_\Theta=2f$. 

To conclude that $f \lambda _{\Theta}^{-1} f ^{\vee} = \lambda _{L_A^\vee}$, we use the fact that $ \lambda_{L_A} f$ is an epimorphism (since it is surjective as an isogeny).  

\end{proof}

A generic element of $\mathcal{J'}$ is a pair $(A'=JC/\langle \alpha _1, \alpha _2 \rangle ^\perp,L_{A'})$ together with a degree $4$ isogeny $f' : A' \longrightarrow JC $ such that $L_{A'}=f'^*(\Theta_C)$. The commutativity of the above diagram means that $\rho(\mathcal{J})=\mathcal{J}'$: indeed, given $(A=JC/\langle \alpha _1, \alpha _2 \rangle ,L_{A}) \in \mathcal{J}$, we have $\rho(A)=A^\vee = JC/\langle \alpha _1, \alpha _2 \rangle ^\perp=A'$. To see that $\rho(L_A)=L_{A'}$, we observe that $f'=f^\vee$ and use Lemma \ref{lem2}. 

It is useful to note that we have the following result :

\begin{lemma}
The pullback by $\lambda _{L_A}$ of $L_{A^\vee}$ is algebraically equivalent to $L_A ^2$.
\end{lemma}
\begin{proof}
The proof is analogous to the previous one: the statement is equivalent to 
\[(\lambda _{L_A})^{\vee} \circ \lambda _{L_A ^\vee } \circ \lambda _{L_A} = 2 \lambda _{L_A}.\]
But since $(\lambda _{L_A})^{\vee}: A \rightarrow A^\vee$ is the same as $\lambda _{L_A}: A \rightarrow A^\vee$, and $\lambda _{L_A} \circ \lambda _{L_A ^\vee } = 2_A$ the equality is straightforward:
\[(\lambda _{L_A})^{\vee} \circ \lambda _{L_A ^\vee } \circ \lambda _{L_A} = \lambda _{L_A} \circ \lambda _{L_A ^\vee } \circ \lambda _{L_A} = 2 \lambda _{L_A}.\]
\end{proof}


\begin{proof}[Proof of Theorem (\ref{main_thm}) point (2)]
Suppose $\mathcal{J} = \mathcal{J}'$. Since elements in $\mathcal{J}$ are of the form $\frac{JC}{\langle \alpha _1, \alpha _2 \rangle}$, for some curve $C$ and some $2$-torsion elements $\alpha _1$ and $\alpha _2$, and elements in $J'$ are of the form $JD/\langle \beta _1, \beta _2 \rangle ^{\perp}$ for some curve $D$ and some $2$-torsion elements $\beta _1$ and $\beta _2$, the equality would imply that for every pair $(JC, \langle \alpha _1, \alpha _2 \rangle)$ in $\mathcal{J}$ we can find another pair $(JD, \langle \beta _1, \beta _2 \rangle)$ such that 
\[\frac{JC}{\langle \alpha _1, \alpha _2 \rangle} = \frac{JD}{\langle \beta _1, \beta _2 \rangle ^{\perp}}.\]

Now consider the diagram
\[\xymatrixcolsep{5pc}\xymatrixrowsep{3pc}
\xymatrix{
JC \ar[d]^{2_{JC}} \ar[r]^-{f_C} &\frac{JC}{\langle \alpha _1, \alpha _2 \rangle }= \frac{JD}{\langle \beta _1, \beta _2 \rangle ^{\perp}} \ar[d]^{\lambda_{L_A}} & JD \ar[l]_-{f_D} \ar[d]^{2_{JD}} \\
JC^{\vee} &\frac{JC^{\vee}}{\langle \alpha _1, \alpha _2 \rangle ^{\perp}}= \frac{JD^{\vee}}{\langle \beta _1, \beta _2 \rangle} \ar[l]_-{f _C ^{\vee}}\ar[r]_-{f _D ^{\vee}} & JD^{\vee}}
\]

Composing $f _C ^{\vee} \circ \lambda _{L_A} \circ f _D$, we obtain an isogeny from $JD$ to $JC^{\vee} \cong JC$, where the isomorphism is given using the principal polarization of $JC$. Computing the degree of this map we find that it is $2^{12}$, that is the product of the degree of the three factorizing maps ($\operatorname{deg} f _D = 2^6$, $\operatorname{deg} \lambda _{L_A} = 2^4$, $\operatorname{deg} f _C ^{\vee} = 2^2$).
Now we use this result:
\begin{teo}[Bardelli, Pirola]\label{BardelliPirola}
If $\chi$ is an isogeny between two Jacobians of dimension $g \geq 4$, and $J$ is generic, than $J \cong J'$ and $\chi$ is the multiplication by an integer.
\end{teo}

Applying Theorem \ref{BardelliPirola} to our case, we get that necessarily the Jacobians $JC$ and $JD$ have to be isomorphic as principally polarized abelian varieties, so, using Torelli Theorem, we get that the curves $C$ and $D$ have to be isomorphic. Moreover, we also get that the isogeny $\chi = f _C ^{\vee} \circ \lambda _{L_A} \circ f _D$ has to be the multiplication by an integer map. But this cannot be, since the multiplication by $m$ map has degree $m^{2 \times 4}=m^8$, thus can never be equal to $2^{12}$. Hence $\mathcal{J} \neq \mathcal{J}'$, which completes the proof.

\end{proof}

\bibliographystyle{plain}
\bibliography{biblio}

\begin{thebibliography}{10}

\bibitem{bardelli1995curves}
Fabio Bardelli, Ciro Ciliberto, and Alessandro Verra.
\newblock Curves of minimal genus on a general abelian variety.
\newblock {\em Compositio Mathematica}, 96(2):115--147, 1995.

\bibitem{bardelli1989curves}
Fabio Bardelli and Gian~Pietro Pirola.
\newblock Curves of genus $g$ lying on a $g$-dimensional {J}acobian variety.
\newblock {\em Inventiones mathematicae}, 95(2):263--276, 1989.

\bibitem{birkenhake2003isomorphism}
Christina Birkenhake and Herbert Lange.
\newblock An isomorphism between moduli spaces of abelian varieties.
\newblock {\em Mathematische Nachrichten}, 253(1):3--7, 2003.

\bibitem{birkenhake2013complex}
Christina Birkenhake and Herbert Lange.
\newblock {\em Complex abelian varieties}, volume 302.
\newblock Springer Science \& Business Media, 2013.

\bibitem{clemens1983double}
Charles~Herbert Clemens.
\newblock Double solids.
\newblock {\em Advances in mathematics}, 47(2):107--230, 1983.

\bibitem{debarre2005complex}
Olivier Debarre.
\newblock {\em Complex tori and abelian varieties}, volume~11.
\newblock American Mathematical Society, 2005.

\bibitem{donagi1984unirationality}
Ron Donagi.
\newblock The unirationality of $\mathcal{A}_5$.
\newblock {\em Annals of Mathematics}, 119:269--307, 1984.

\bibitem{donagi1992fibers}
Ron Donagi.
\newblock {The fibers of the Prym map. Curves, Jacobians, and abelian
  varieties}.
\newblock {\em Contemporary Mathematics}, 136:55--125, 1990.

\bibitem{hulek1994surfaces}
K.~Hulek and G.~K. Sankaran.
\newblock The kodaira dimension of certain moduli spaces of abelian surfaces.
\newblock {\em Compositio Mathematica}, 90(1):1--35, 1994.

\bibitem{mumford1983kodaira}
David Mumford.
\newblock On the {K}odaira dimension of the {S}iegel modular variety.
\newblock In {\em Algebraic geometry, open problems}, pages 348--375. Springer,
  1983.

\bibitem{naranjo2016degree}
Juan~Carlos Naranjo and Angela Ortega.
\newblock The degree of the {P}rym map of ramified coverings, remaning cases.
\newblock {\em in preparation}, 2016.

\bibitem{pantazis1986prym}
Stefanos Pantazis.
\newblock Prym varieties and the geodesic flow on ${SO}(n)$.
\newblock {\em Mathematische Annalen}, 273(2):297--315, 1986.

\bibitem{tai1982kodaira}
Yung-Sheng Tai.
\newblock On the {K}odaira dimension of the moduli space of abelian varieties.
\newblock {\em Inventiones mathematicae}, 68(3):425--439, 1982.

\end{thebibliography}

\end{document}